\documentclass[10pt]{amsart}

\usepackage{amsfonts}
\usepackage{amssymb}
\usepackage{graphicx}
\usepackage{amsmath, amsthm, latexsym, bm, bbm, mathrsfs}
\usepackage[all]{xy}
\usepackage{hyperref}
\usepackage[margin=1.2in]{geometry}
\usepackage{color}
\usepackage{subfig}
\usepackage{float}

\def \C{\mathbb{C}}
\def \Z{\mathbb{Z}}
\def \R{\mathbb{R}}

\def \P{\mathbb{P}}

\def \h{\textbf{h}}
\def \O{\mathcal{O}}

\def \Spec{\operatorname{Spec}}

\def \Span{\operatorname{span}}
\def \interior{\operatorname{Int}}
\def \symmetry{\operatorname{Sym}}
\def \PL{\operatorname{PL}}

\def \codim{\operatorname{codim}}

\def \V{\mathcal{V}}
\def \E{\mathcal{E}}

\def \u{\bm{u}}
\def \L{\mathcal{L}}

\newcommand*\abs[1]{\left\lvert#1\right\rvert}
\newcommand*\dotpro[2]{\langle{#1} , {#2}\rangle}

\theoremstyle{plain}
\newtheorem{theorem}{Theorem}[section]
\newtheorem{lemma}[theorem]{Lemma}
\newtheorem{prop}[theorem]{Proposition}
\newtheorem{corollary}[theorem]{Corollary}

\theoremstyle{definition}
\newtheorem{example}[theorem]{Example}
\newtheorem{definition}[theorem]{Definition}
\newtheorem{remark}[theorem]{Remark}

\begin{document}

\title{Note on Euler characteristic of a toric vector bundle}
\author{Suhyon Chong}
\address{Department of Mathematics, University of Pittsburgh,
Pittsburgh, PA, USA.}
\email{suc86@pitt.edu}

\author{Shaoyu Huang}
\address{Department of Mathematics, University of Pittsburgh,
Pittsburgh, PA, USA.}
\email{shh123@pitt.edu}

\author{Kiumars Kaveh}
\address{Department of Mathematics, University of Pittsburgh, Pittsburgh, PA, USA.}
\email{kaveh@pitt.edu}

\date{\today}

\begin{abstract}
A convex chain is a finite integer linear combination of indicator functions of convex polytopes. Khovanskii-Pukhlikov extend the Ehrhart theory of convex lattice polytopes to the setting of convex chains. Extending the relationship between equivariant line bundles on projective toric varieties and virtual lattice polytopes, we associate a lattice convex chain to a torus equivariant vector bundle on a toric variety and show that sum of values of this convex chain on lattice points gives the Euler characteristic of the bundle.  
\end{abstract}

\maketitle
\tableofcontents

\section{Introduction}
The purpose of this short note is to illustrate that the Khovanskii-Pukhlikov theory of convex chains and multi-valued support functions (\cite{Khovanskii-Pukhlikov-1, Khovanskii-Pukhlikov-2}) is the right gadget to extend the Ehrhart theory of toric line bundles, which concerns convex lattice polytopes, to toric vector bundles. We work over $\C$ but all the results hold over any algebraically closed field of characteristic $0$. 

We recall that a \emph{toric vector bundle}, is a vector bundle on a toric variety equipped with a linear torus action lifting the one on the base toric varity.  

One of the central facts in the theory of toric varieties is that the dimension of the space of global sections of an ample toric line bundle is the number of lattice points in the corresponding Newton polytope (also called the momentum polytope). We use the Khovanskii-Pukhlikov theory of convex chains to extend this well-known fact to toric vector bundles.

Let $M \cong \Z^n$ be a lattice and let $M_{\R}=M\otimes_{\Z}\R \cong \R^n$. Let $\mathscr{P}_M$ be the set of all convex lattice polytopes in $M_{\R}$. We think of $M$ as the lattice of characters of a torus $T \cong (\C^*)^n$.

Let $P \in \mathscr{P}_M$ be a full dimensional lattice polytope. The Ehrhart function $L_P$ of $P$ is defined by $L_P(t) = |tP \cap \Z^n|$, for positive integers $t > 0$. A well-known theorem of Ehrhart states that $L_P$ is the restriction of a polynomial to positive integers (see Section \ref{subsec-Ehrhart}). The Ehrhart function has an interpretation in terms of algebro-geometry of toric varieties: Let $\Sigma$ be the normal fan of $P$, the polytope $P$ determines a (torus equivariant) ample line bundle $\mathscr{L}_P$ on the toric variety $X_\Sigma$ such that, for any integer $t>0$, the dimension of space of global sections $H^0(X_\Sigma, \mathscr{L}_P^{\otimes t})$ is equal to $L_P(t)$.

In \cite{Khovanskii-Pukhlikov-1}, Khovanskii-Pukhlikov extend the Ehrhart theory of lattice convex polytopes to lattice convex chains. A \emph{convex chain} is a function $\alpha:M_{\R}\to \Z$ of the form $\alpha=\sum_i n_i\mathbbm{1}_{P_i}$, where $n_i\in\Z$ and $P_i\in\mathscr{P}_{M_\R}$ are convex polytopes. Here $\mathbbm{1}_P$ denotes the indicator function of a polytope $P$. If all the $P_i$ are lattice polytopes, then we call $\alpha$ a \emph{lattice convex chain}.
One of the main results in \cite{Khovanskii-Pukhlikov-1} is that the function $\alpha \mapsto \Sigma_{u \in M} \alpha(u)$ behaves in a polynomial fashion for lattice convex chains. This far extends the polynomiality result for the Ehrhart function of a lattice polytope. The algebra of convex chains is closely related to McMullen's polytope algebra (see \cite{McMullen, Brion}).


For a toric vector bundle $\E$ on a $T$-toric variety $X_\Sigma$ and a character $u \in M$, we define the equivariant Euler characteristic of $\E$ at $u$, denoted by $\chi(\E)_u = \chi(X_\Sigma, \E)_u$, to be the Euler characteristic of the sheaf of $u$-weight sections of $\E$. Clearly, $\chi(\E) = \sum_{u \in M} \chi(\E)_u$.

Extending the relationship between equivariant line bundles on projective toric varieties and virtual lattice polytopes, we associate a lattice convex chain $\alpha_\E$ to a toric vector bundle $\E$ and show that the values of the convex chain $\alpha_\E$ give the equivariant Euler characteristic of the bundle (Section \ref{sec:cc-main}). The idea of associating a convex chain to a toric vector bundle is due to Askold Khovanskii.

\begin{theorem}
Let $\E$ be a toric vector bundle on a projective toric variety $X_\Sigma$. The equivariant Euler characteristic of $\E$ is given by
$$\chi(\E)_u = \alpha_\E(u), ~~~~~\forall u \in M.$$
In particular,
$$\chi(\E) = \sum_{u\in M}\alpha_\E(u).$$    
\end{theorem}

In Section \ref{sec:example}, we explicitly compute some examples.

\begin{remark}
In the forthcoming paper \cite{Chong-Kaveh}, among other things, the results of the present note are generalized to \emph{tropical vector bundles}. Also the convex chain associated to a tropical vector bundle makes an appearance in \cite[Section 6]{KM-trop-vbs}.
\end{remark}

\begin{remark}
We point out that, in general, the convex chain $\alpha_\E$ contains different information than a parliament of polytopes of $\E$ introduced in \cite{DJS}. The convex chain $\alpha_\E$ encodes the equivariant Euler characteristic of $\E$ while the parliament of polytopes encodes the space of global sections of $\E$. 
\end{remark}

\begin{remark}
One can give a slightly different construction of the convex chain $\alpha_\E$ as follows. Given a smooth projective toric variety $X_\Sigma$, the equivariant $K$-ring of $X_\Sigma$ can be identified with a subring of the ring of convex chains (see \cite[Theorem A.10]{EHL}). Then for a toric vector bundle $\E$ on $X_\Sigma$, one writes its equivariant $K$-class as a sum/difference of toric line bundles. It can be shown (but it is not obvious) that this construction gives the same convex chain as $\alpha_\E$.

In fact, in the paper \cite[Section 4]{Chong-Kaveh}, the equivariant $K$-class of a tropical vector bundle is represented as an alternating sum of equivariant $K$-classes of toric line bundles. This in particular applies to toric vector bundles.


\end{remark}

\subsection*{Acknowledgment}
We would like to thank Askold Khovanskii for suggesting the association of a multi-valued support function and hence a convex chain to a toric vector bundle. We also thank Amin Gholampour, Klaus Altmann and Matt Larson for useful email correspondence. In particular, the proof of Lemma \ref{lem-pull-back-inv} was suggested to us by Gholampour and Altmann. The third author is partially supported by National Science Foundation Grant DMS-210184 and a Simons Collaboration Grant.

\subsection*{Notation}
\begin{itemize}
\item $T=(\C^*)^n$ denotes an algebraic torus of dimension $n$ with $M$ and $N$ its character and cocharacter lattices respectively. The pairing between them is denoted by $\dotpro{\cdot}{\cdot}:M \times N \rightarrow \Z$. We let $M_{\R}=M\otimes\R$ and  $N_{\R}=N\otimes\R$ to be the corresponding $\R$-vector spaces.
\item $X_\Sigma$ is a $T$-toric variety corresponding to a fan $\Sigma$ in $N_{\R}$. The support of $\Sigma$ is denoted by $\abs{\Sigma}$. In this paper, we often assume that $X_\Sigma$ is projective.
\item $U_\sigma$ is an affine $T$-toric variety corresponding to a cone $\sigma\in\Sigma$. $\sigma^\vee$ is the dual cone of $\sigma$.
\item $P$ usually denotes a lattice convex polytope in $M_\R$. The $h_P: N_\R \to \R$ is the support function of $P$, and $L_P: \Z \to \Z$ is its Ehrhart function/polynomial.
\item $\E$ is a toric vector bundle on a toric variety $X_\Sigma$.
\item $\h_\E$ is the multi-valued support function associated to a toric vector bundle $\E$ (Section \ref{subsec-cc-tvb}).
\item $\alpha_\E$ is the convex chain corresponding to a toric vector bundle $\E$ (Section \ref{subsec-cc-tvb}).
\item $\PL(N, \Z)$ denotes the semifield of all piecewise linear functions, with respect to some complete fan, on the vector space $N_\R$ which have integer values on the lattice $N$. Moreover, we add a unique ``infinity element'' $\infty$ to $\PL(N, \Z)$ which is greater than any other element.
\end{itemize}

\section{Background}
\subsection{Ehrhart theory}  \label{subsec-Ehrhart}
In this section we review well-known results about the Ehrhart polynomial of a lattice polytope and relation with line bundles on toric varieties. 

Let $P$ be a $d$-dimensional lattice polytope in ${\R}^n$, that it, its vertices are points of the lattice ${\Z}^n$. For any positive integer $t$, let $tP$ be the $t$-times dilation of $P$, and let 
\begin{equation*}
    L_P(t) = \abs{tP\cap {\Z}^n}
\end{equation*}
be the number of lattice points contained in the polytope $tP$.

Ehrhart showed in the 1960's that $L_P(t)$ is a rational polynomial of degree $d$ in $t$ \cite{Ehrhart-poly}. We call $L_P$ the \emph{Ehrhart polynomial} of $P$. The following beautiful theorem describes relation between the Ehrhart polynomial and number of lattice points in the relative interior of a polytope.

\begin{theorem}[Ehrhart–Macdonald reciprocity]   \label{th-Ehrhart-recip}
    Let $P^\circ$ be the relative interior of $P$, then
    \begin{equation*}
        L_{P^\circ}(t) = (-1)^dL_P(-t),
    \end{equation*}
where $L_{P^\circ}(t) = |tP^\circ \cap \Z^n|$.  \end{theorem}

Let $\mathcal{F}$ be a coherent sheaf on a complete variety $X$. Recall that the \emph{Euler characteristic} of $\mathcal{F}$ is defined as
\begin{equation*}
    \chi(X, \mathcal{F}) = \sum_{i\ge0}(-1)^ih^i(X, \mathcal{F}),
\end{equation*}
where $h^i(X, \mathcal{F})=\dim H^i(X, \mathcal{F})$.

For a Cartier divisor $D$ on $X$ let $\mathcal{O}_X(D)$ denote the sheaf of rational functions associated to $D$. One knows that $D \mapsto \chi(X, \mathcal{O}_X(D))$ is a polynomial on the Picard group of $X$ (when $X$ is smooth this follows from the Hirzebruch-Riemann-Roch theorem). For an ample line bundle on a toric variety $X_\Sigma$, there is an elegant connection with the Ehrhart polynomial. Firstly, one knows the following \cite[Section 9.2]{Cox-Little-Schenck}:
\begin{theorem}\label{thm-dem-van}
 For an ample divisor $D$ on $X_\Sigma$, $h^i(X_\Sigma, \mathcal{O}_{X_\Sigma}(D))=0$ for all $i>0$. Hence $\chi(X_\Sigma, \mathcal{O}_{X_\Sigma}(D)) = h^0(X_\Sigma, \mathcal{O}_{X_\Sigma}(D))$.
\end{theorem}
    
Let $D=\sum_{\rho\in\Sigma(1)}a_\rho D_\rho$ be a torus invariant divisor on $X_\Sigma$. To $D$ there corresponds a polytope \[P_D=\{m\in M_{\R}\mid\dotpro{m}{v_\rho}\ge-a_\rho,\forall\rho\in\Sigma(1)\},\] where $v_\rho$ is the minimal integral generator of a ray $\rho$.
One then has a combinatorial description of the space of global sections: 
\begin{equation*}
    \Gamma(X_\Sigma, \mathcal{O}_{X_\Sigma}(D))=\bigoplus_{m\in P_D\cap M}\C\cdot\chi^m.
\end{equation*}
Here $\chi^m$ is the character of the torus $T$ corresponding to $m \in M$ regarded as a weight section of $\mathcal{O}_{X_\Sigma}(D)$ of weight $m$. It follows that the dimension of the space of global sections is equal to the number of lattice points in $P_D$. 

From polynomiality of the Euler characteristic and the Ehrhart function one obtains the following.
\begin{corollary}\label{cor:chi-ample}
    For an ample divisor $D$ on $X_\Sigma$,
    \begin{equation*}
        \chi(X_\Sigma, \mathcal{O}_{X_\Sigma}(tD))= L_{P_D}(t), ~~\forall t \in \Z
    \end{equation*}
\end{corollary}

\begin{example}
When $t\ge0$, we know $H^0(\P^n,\mathcal{O}_{\P^n}(t))$ consists of homogeneous polynomials in $n+1$ variables of degree $t$, which gives us
\begin{equation*}
    \chi(\P^n,\mathcal{O}_{\P^n}(t))=h^0(\P^n,\mathcal{O}_{\P^n}(t))=\binom{t+n}{n}.
\end{equation*}

On the other hand, let $\Delta_n$ be the standard $n$-simplex. We know $L_{\Delta_n}(t)$ is equal to the number of non-negative integer solutions of
the inequality $x_1+x_2+\cdots+x_n\le t$, hence:
\begin{equation*}
    L_{\Delta_n}(t)=\abs{t\Delta_n\cap\Z^n}=\binom{t+n}{n},
\end{equation*}
which verifies that $\chi(\P^n,\mathcal{O}_{\P^n}(t))=L_{\Delta_n}(t)$.
\end{example}

Finally, we recall the Serre duality theorem.
\begin{theorem}[Serre duality]
    Let $X$ be a smooth complete variety of dimension $n$ over $\C$ and let $\mathcal{E}$ be a vector bundle on $X$. Then there is a natural isomorphism 
    \begin{equation*}
        H^i(X,\mathcal{E})\simeq H^{n-i}(X,K_X\otimes\mathcal{E}^*)^*,
    \end{equation*}
where $K_X:=\bigwedge^n(T^*X)$ is the canonical line bundle.
\end{theorem}
This gives us 
\begin{equation*}
    \begin{split}
        \chi(X,\E)&=\sum_{i}(-1)^ih^i(X,\mathcal{E})=\sum_{i}(-1)^ih^{n-i}(X,K_X\otimes\E^*)\\
        &=(-1)^n\sum_{i}(-1)^{i}h^{i}(X,K_X\otimes\mathcal{E}^*)=(-1)^n\chi(X,K_X\otimes\E^*).
    \end{split}
\end{equation*}

Let $\E=\mathcal{O}_{X_\Sigma}(D+K_{X_\Sigma})$, we have
\begin{equation*}
    \chi(X_\Sigma,\mathcal{O}_{X_\Sigma}(D+K_{X_\Sigma}))=(-1)^n\chi(X_{X_\Sigma},\mathcal{O}_{X_\Sigma}(-D)).
\end{equation*}

Since $K_{X_\Sigma}=-\sum_{\rho\in\Sigma(1)}D_\rho$, the polytope $P_{D+K_{X_\Sigma}}$ comes from shrinking $P_D$ by $1$ unit along each facet normal. 
Putting all these together, we get a proof of Theorem \ref{th-Ehrhart-recip}.

\subsection{Khovanskii-Pukhlikov theory of convex chains}  
\label{subsec-KhPu}
As mentioned in the introduction, Khovanskii-Pukhlikov give an elegant generalization of the Ehrhart theory for lattice convex polytopes to the so-called \emph{convex chains} (see \cite{Khovanskii-Pukhlikov-1}). 
In this section we give a brief review of their theory.

Let $\mathscr{P}_{M_\R}$ (respectively $\mathscr{P}_M$) be the set of all the convex polytopes in $M_{\R}$ (respectively all the lattice convex polytopes in $M_\R$). For $P, Q \in \mathscr{P}_{M_\R}$, the Minkowski sum is a new convex polytope defined as
\begin{equation*}
    P\oplus Q=\{x+y\mid x\in P, y\in Q\}.
\end{equation*}
The Minkowski sum satisfies the cancellation property: if $P\oplus Q = P' \oplus Q$ then $P=P'$. The set $\mathscr{P}_{M_\R}$ together with $\oplus$ is a commutative monoid with the origin $\{0\}$ as the identity element. Since the monoid $\mathscr{P}_{M_\R}$ is cancellative, it can be extended to a group, namely its Grothendieck group $\mathscr{P}^*_{M_\R}$, which is called the group of \emph{virtual polytopes}. The inverse of a $P \in \mathscr{P}_{M_\R}$ in $\mathscr{P}^*_{M_\R}$ is denoted by $-P$.

\begin{definition}
    A \emph{convex chain} is a function $\alpha:M_{\R}\to \Z$ of the form $\alpha=\sum_i n_i\mathbbm{1}_{P_i}$, where $n_i\in\Z$ and $P_i\in\mathscr{P}_{M_\R}$. Here, for a subset $P \subset M_\R$, $\mathbbm{1}_P$ denotes the indicator function of $P$ which attains $1$ on $P$ and $0$ outside $P$. 
    
    The sum $\sum_i n_i$ is called the \emph{degree} of $\alpha$, which is denoted by $\deg\alpha$. We say that $\alpha$ is a lattice convex chain if all the $P_i$ are lattice convex polytopes. 
\end{definition}
The set of convex chains with addition of functions is a group. We denote the additive group of convex chains (respectively lattice convex chains) by $\Z(M_\R)$ (respectively by $\Z(M)$). 

\begin{remark}  \label{rem-polytope-alg}
The algebra of convex chains is closely related (but not identical) to McMullen's polytope algebra (see \cite{McMullen, Brion}).   
\end{remark}

\begin{prop}
Minkowski summation of polytopes extends in a unique way to a bilinear operation on $\Z(M_\R)$:
$$*:\ \Z(M_\R) \times \Z(M_\R) \to \Z(M_\R),$$
such that for any two convex polytopes $A$, $B$: $$\mathbbm{1}_A * \mathbbm{1}_B = \mathbbm{1}_{A\oplus B}.$$ 
\end{prop}
We take $*$ as the multiplication on $\Z(M)$. Under this multiplication, $\Z(M)$ becomes a commutative ring with identity $\mathbbm{1}_{\{0\}}$, where $0$ is the origin in $M_{\R}$.

The \emph{support function} of a polytope $P\in\mathscr{P}_M$ is a piecewise-linear convex function $h_P:N_{\R}\to\R$,

\begin{equation*}
    h_P(x)=\max_{m\in P}\dotpro{m}{x}.
\end{equation*}

One verifies that for $P,Q\in\mathscr{P}_M$,
    \begin{equation*}
        h_{P\oplus Q}=h_P + h_Q.
    \end{equation*}

\begin{definition}
    Let $R=P\oplus(-Q)$ be a virtual polytope. The support function of $R$ is defined as $h_R=h_P-h_Q$.
\end{definition}

Under this definition, the set of support functions becomes a group.

\begin{theorem}\label{thm-vir-supp}
    The group of virtual polytopes $\mathscr{P}^*_M$ and the group of support functions are isomorphic.
\end{theorem}

\begin{definition}
    The \emph{(multi-valued) support function} of a convex chain $\alpha=\sum_i n_i\mathbbm{1}_{P_i}\in \Z(M)$ is a function
    \begin{equation*}
        \begin{split}
            \h_\alpha: N_{\R}\to& \Z[\R], \\
            x\mapsto& \sum_i n_i [h_{P_i}(x)].
        \end{split}
    \end{equation*}
\end{definition}

The addition in $\Z[\R]$ is the formal sum and the multiplication in $\Z[\R]$ comes from the usual sum in $\R$.

\begin{prop}
    The association of a convex chain to its support function is a ring isomorphism between $\Z(M)$ and the ring of all piecewise-linear functions from $N_{\R}$ to $\Z[\R]$.
\end{prop}

The surjectivity comes from the fact that a piecewise-linear function is the difference of two piecewise-linear convex functions.

For a polytope $P$ in $M_{\R}$, we have 
\begin{equation*}
    (-1)^{\dim P}\mathbbm{1}_{P^\circ}=\sum_{\Delta\in\Gamma(P)}(-1)^{\dim \Delta}\mathbbm{1}_\Delta,
\end{equation*}
where $\Gamma(P)$ is the set of all faces of $P$, including $P$ itself.

The following theorem gives us an ``inverse'' of convex polytopes.

\begin{theorem}[Minkowski inversion]  \label{th-Minkowski-inversion}
    Let $P\in\mathscr{P}_M$, then 
    \begin{equation*}
        (-1)^{\dim P}\mathbbm{1}_{(\symmetry P)^\circ}*\mathbbm{1}_P=\mathbbm{1}_{\{0\}},
    \end{equation*}
    where $\symmetry P$ is the central symmetry of $P$ with respect to the origin.
\end{theorem}

From the above theorem, we have 
\begin{equation*}
    (\mathbbm{1}_P)^{-1} = \sum_{\Delta\in\Gamma(\symmetry P)}(-1)^{\dim \Delta}\mathbbm{1}_\Delta.
\end{equation*}

From Euler characteristic for convex polytopes, we can see from the above formula that
\begin{equation*}
    \deg(\mathbbm{1}_P)^{-1} = \deg\sum_{\Delta\in\Gamma(\symmetry P)}(-1)^{\dim \Delta}\mathbbm{1}_\Delta=
    \sum_{\Delta\in\Gamma(\symmetry P)}(-1)^{\dim \Delta}=1,
\end{equation*}
which corresponds to invertible elements $\alpha\in\Z(M)$ satisfying $\deg\alpha=1$.

The convex chain associated to a virtual polytope $P\oplus(-Q)$ is given by $\mathbbm{1}_P*(\mathbbm{1}_Q)^{-1}$.

\subsection{Toric vector bundles}
\label{subsec-TVBs}
We recall that a \emph{toric vector bundle} on a toric variety $X_\Sigma$ is a vector bundle $\mathcal{E}$ together with a $T$-linearization, i.e. there is an action of $T$ on $\E$ that lifts the $T$-action on $X_\Sigma$ such that the action on each fiber is linear.

In \cite{Klyachko}, Klyachko gives a classification of toric vector bundles on toric varieties, which is an extension of the classification of equivariant line bundles on toric varieties by integral piecewise linear functions.

Consider a toric vector bundle $\E$ on a toric variety $X_\Sigma$. The action of $T$ induces a $T$-action on the sections $\Gamma(U_\sigma,\mathcal{E})$ for any cone $\sigma\in\Sigma$:
\begin{equation*}
    (t\cdot s)(x)=t(s(t^{-1}x)),
\end{equation*}
where $s\in\Gamma(U_\sigma,\mathcal{E}),x\in U_\sigma$. We have a decomposition into $T$-eigenspaces $$\Gamma(U_\sigma,\mathcal{E})=\bigoplus_{u\in M}\Gamma(U_\sigma,\mathcal{E})_u,$$ where $\Gamma(U_\sigma,\mathcal{E})_u=\{s\in\Gamma(U_\sigma,\mathcal{E})\mid t\cdot s = \chi^u(t)s,\ \forall t\in T\}$.

We fix a point $x_0$ in the open orbit in $X_\Sigma$. Let $E=\mathcal{E}_{x_0}$ be the fiber of $\mathcal{E}$ over $x_0$. For $s\in\Gamma(U_\sigma,\mathcal{E})_u$, since
\begin{equation*}
    s(t x_0)=\chi^{-u}(t)t(s(t^{-1}tx_0))=\chi^{-u}(t)t(s(x_0)),
\end{equation*}
$s$ is determined by its value at $x_0$. This gives an injection $\Gamma(U_\sigma,\mathcal{E})_u\overset{\iota}{\hookrightarrow} E$ via $\iota(s)=s(x_0)$. Let the image of $\Gamma(U_\sigma,\mathcal{E})_u$ in $E$ be $E^\sigma_u$.

For $u'\in\sigma^\vee\cap M$, multiplication by $\chi^{u'}$ gives an injection $\Gamma(U_\sigma,\mathcal{E})_u\hookrightarrow \Gamma(U_\sigma,\mathcal{E})_{u-u'}$.
Moreover, the multiplication map by $\chi^{u'}$ commutes with the evaluation at $x_0$ and hence induces an inclusion $E_u^\sigma\subseteq E^\sigma_{u-u'}$. If $u' \in \sigma^\perp$, where $\sigma^\perp=\{u\in M_\R\mid \dotpro{u}{x}=0,\ \forall x\in\sigma\}$, then these maps are isomorphisms and thus $E_u^\sigma$ depends only on the class $[u] \in M_\sigma=M/(\sigma^\perp\cap M)$. 

For any $\rho\in\Sigma(1)$, we define
\begin{equation*}
    E^\rho(i)=E_u^\rho,
\end{equation*}
where $\dotpro{u}{v_\rho}=i$. Since $E^\sigma_u$ only depends on $[u]\in M_\sigma$, $E^\rho(i)$ is well-defined. Thus we have a decreasing filtration of $E$ $$\cdots\supseteq E^\rho(i-1)\supseteq E^\rho(i)\supseteq E^\rho(i+1)\cdots.$$

An important step in the classification of toric vector bundles is that a toric vector bundle over an affine toric variety is \emph{equivariantly trivial}. That is, it decomposes $T$-equivariantly as a sum of trivial line bundles (see \cite[Proposition 2.1.1]{Klyachko}):
\begin{prop}\label{prop:toric-affine-equiv-trivial}
    Every toric vector bundle $\mathcal{E}$ with rank $r$ on an affine toric variety $U_\sigma$ splits equivariantly into a sum of trivial line bundles:
    \begin{equation*}
\mathcal{E}=\bigoplus_{j=1}^r\mathcal{L}_{u_j},
    \end{equation*}
    where $\mathcal{L}_{u_j}$ denotes the trivial toric line bundle on $U_\sigma$ where $T$ acts on it via the character $u_j$.
\end{prop}
The above shows that, for each $\sigma \in \Sigma$, the filtrations $(E^\rho(i))_{i \in \Z}$ satisfy the following compatibility condition: There is a decomposition of $E$ into a direct sum of $1$-dimensional subspaces indexed by a finite multiset $\bm{u}_\E(\sigma) \subset M_\sigma$:
\begin{equation*}
    E = \bigoplus_{[u] \in \bm{u}_\E(\sigma)} E_{[u]},
\end{equation*}
such that for any ray $\rho \in \sigma(1)$ we have:
\begin{equation}  \label{equ-Klyachko-comp-condition}
E^\rho(i) = \sum_{\langle u, v_\rho \rangle \geq i}  E_{[u]}.
\end{equation}

We call a collection of decreasing $\Z$-filtrations $\{(E^\rho(i))_{i \in \Z} \mid \rho \in \Sigma(1) \}$ satisfying the condition \eqref{equ-Klyachko-comp-condition} a \emph{compatible collection of filtrations}. 
The following is Klyachko's classification theorem for toric vector bundles (\cite[Theorem 2.2.1]{Klyachko}). 
\begin{theorem}[Klyachko]   \label{th-Klyachko}
The category of toric vector bundles on $X_\Sigma$ is naturally equivalent to the category of compatible filtrations on finite dimensional $\C$-vector spaces.
\end{theorem}

\subsection{Toric vector bundles as piecewise linear valuations}  \label{subsec-tvbs-val}

In \cite{KM-TVBs-valuations}, toric vector bundles, up to pull back by toric birational morphisms, are classified by valuations on a vector space $E$ and with values in the semifield of integral piecewise linear functions. This is a generalization of classification of toric line bundles by integral piecewise linear functions. In this section, we recall this classification.

Let $\PL(N, \Z)$ denote the set of all piecewise linear functions with respect to some complete fan, on the vector space $N_\R$ which have integer values on the lattice $N$. Moreover, we add a unique ``infinity element'' $\infty$ to $\PL(N, \Z)$ which is greater than any other element. We regard it as the function which assigns value infinity to all points in $N_\R$. The set $\PL(N, \Z)$ is an idempotent semifield with operations of taking minimum and addition of functions. We refer to $\PL(N, \Z)$ as the \emph{piecewise linear semifield}.

There is a natural partial order on $\PL(N, \Z)$ where $h_1\le h_2$ if for all $x$ in $N_{\R}$, $h_1(x)\le h_2(x)$. Following \cite{KM-TVBs-valuations}, we define the notion of a piecewise linear valuation, which is an extension of the notion of a valuation on a vector space.

\begin{definition}[Piecewise linear valuation]  \label{def-PL-val}
Let $E$ be a finite dimensional $\C$-vector space. A map $\V: E \to\PL(N, \Z)$ is a \emph{piecewise linear valuation} if:
\begin{enumerate}
\item[(1)] For any $e \in E$ and any $0 \neq c \in \C$ we have $\V(c e) = \V(e)$.
\item[(2)] For any $e_1, e_2 \in E$ we have
$\V(e_1+e_2) \geq \min(\V(e_1), \V(e_2))$. 
\item[(3)] $\V(e) = \infty$ if and only if $e = 0$.
\end{enumerate}
Unlike valuations with values in a totally ordered abelian group such as $\Z$, the image of $\V$, in general, might be infinite. If the image of $\V$ is finite we call $\V$ a \emph{finite piecewise linear valuation}.
\end{definition}

Let $\Sigma_1$, $\Sigma_2$ be fans with the same support. We say that toric vector bundles $\E_i$ over  $X_{\Sigma_i}$, $i=1, 2$, are \emph{equivalent} if there is a common refinement $\Sigma$ of the $\Sigma_i$ such that $\pi_1^*(\E_1) \cong \pi_2^*(\E_2)$, where $\pi_i: X_\Sigma \to X_{\Sigma_i}$ is the toric blow-up corresponding to refinement of $\Sigma_i$ to $\Sigma$. 

The following is from \cite[Section 3.3 ]{KM-TVBs-valuations}:
\begin{theorem}[Toric vector bundles as piecewise linear valuations]   \label{th-intro-tvb-preval}
The equivalence classes of toric vector bundles $\E$ on a complete toric variety $X_\Sigma$ , with $E$ as the fiber over the distinguished point $x_0$, with respect to the above equivalence (pull back by toric morphisms), are in one-to-one correspondence with finite piecewise linear valuations $\V: E \to \PL(N, \Z)$.
\end{theorem}

\section{Convex chains and toric vector bundles}\label{sec:cc-main}
In this section, we define the convex chain associated to a toric vector bundle as well as a piecewise linear valuation. We show that the Euler characteristic of a toric vector bundle is equal to the sum of values over the lattice points of its associated convex chain.

\subsection{Convex chain associated to a toric vector bundle}  \label{subsec-cc-tvb}
Let $\E$ be a toric vector bundle on a complete toric variety $X_\Sigma$. From Proposition \ref{prop:toric-affine-equiv-trivial}, for any cone $\sigma\in\Sigma$, we have a multiset of characters $\bm{u}_\E(\sigma)=\{ [u_1], \ldots, [u_r]\}$. First, we define the multi-valued support function $\h_\E$ by:
\begin{equation*}
    \h_{\E}(x)=[\langle u_1, x \rangle]+\cdots+[\langle u_r, x \rangle] \in \Z[\R], \quad x \in \sigma.
\end{equation*}
For each $1 \leq i \leq r$, let us define the piecwise linear function $h_i$ as follows: for any cone $\sigma \in \Sigma$ and any $x \in \sigma$, let $h_i(x)$ be the $i$-th smallest number in the multiset $\{ \langle u_1, x \rangle,\ldots, \langle u_r, x \rangle \}$. Thus: $$h_1(x) \leq \cdots \leq h_r(x).$$
Then $\h_\E = [h_1]+\cdots+[h_r] \in \Z[\R]$ which shows that $\h_\E$ is indeed an $r$-valued support function.

\begin{definition}  \label{def:convex-chain-tvb}
We define the convex chain $\alpha_\E$ associated to the toric vector bundle $\E$ to be the convex chain associated to the multi-valued support function $\h_\E$.
\end{definition}

\subsection{Convex chain associated to a piecewise linear valuation}\label{subsec-con-chain-val}
Let $\V: E \to \PL(N, \Z)$ be a piecewise linear valuation (Definition \ref{def-PL-val}). In this section we (canonically) associate to $\V$ a convex chain $\alpha_\V$. If $\V$ is the piecewise linear valuation associated to a toric vector bundle $\E$, the convex chain $\alpha_\V$ coincides with the convex chain $\alpha_\E$ constructed in Section \ref{subsec-cc-tvb}.


Let $x \in N_\R$. Then $\V_x = \V(\cdot)(x)$ is an $\R$-valued valuation on $E$ and hence gives an $\R$-filtration $(E^x(i))_{i \in \R}$, where $E^x(i)=\{e\in E\mid\V(e)(x)\ge i\}$. If $\V_x(E \setminus \{0\})=\{a_1>\cdots>a_k\}$ then we have a flag:
\begin{equation*}
    F_{\V_x,\bullet} = \Big( \{0\} = F_0  \subsetneqq F_1\subsetneqq \cdots \subsetneqq F_k=E\Big),
\end{equation*}
where $F_i=E^x(a_i)$ and we let $a_0=\infty$.
That is, $\V_x$ gives a labeled flag of subspaces in $E$. The value of the multi-valued support function $\h_\V:N_{\R}\to\Z[\R]$ at $x$ is constructed from $\V_x$ in the following way. The value of $\h_\V$ at $x$ is the multiset $\{a_1, \ldots, a_k\}$ where each $a_i$ is repeated $\dim(F_i / F_{i-1})$ times. 

\begin{definition}  \label{def:convex-chain-PL-val}
We let $\alpha_\V$ to be the convex chain associated to the multi-valued support function $\h_\V$. 
\end{definition}

\subsection{Euler characteristic as sum of values of a convex chain over lattice points}   \label{subsec-Euler-char-tvb}
Let $\E$ be a vector bundle on a complete variety $X$. Its \emph{Euler characteristic} $\chi(\E)$ is defined to be 
\begin{equation*}
    \chi(\E)=\sum_{i}(-1)^i\dim H^i(X,\E).
\end{equation*}

Moreover, if $X = X_\Sigma$ is a $T$-toric variety and $\E$ is a toric vector bundle, then we have a $T$-weight space decomposition for each cohomology: 
\begin{equation*}
    H^i(X_\Sigma,\mathcal{E})=\bigoplus_{u\in M}H^i(X_\Sigma,\mathcal{E})_u.
\end{equation*}

For each character $u \in M$, we define the \emph{equivariant Euler characteristic at $u$} as 
\begin{equation*}
    \chi(\mathcal{E})_u=\sum_{i}(-1)^i\dim H^i(X_\Sigma,\mathcal{E})_u.
\end{equation*}
We refer to the function $u \mapsto \chi(\E)_u$ as the \emph{equivariant Euler characteristic of $\E$}.


Let $D$ be an invariant ample divisor on $X_\Sigma$ and $\L=\O_{X_\Sigma}(D)$. Then the convex chain $\alpha_\L$ coincides with the indicator function of the polytope $P_D$. In fact, from Theorem \ref{thm-dem-van}, for any $u \in M$:
    \begin{equation*}
        \chi(\mathcal{L})_u=\sum_{i}(-1)^i\dim H^i(X_\Sigma,\mathcal{L})_u=h^0(X_\Sigma,\mathcal{L})_u = \mathbbm{1}_{P_D}(u) = \alpha_\L(u).
    \end{equation*}

The following proposition extends the above to non-ample divisors. This appears in the work of Khovanskii-Pukhlikov (\cite[p. 794]{Khovanskii-Pukhlikov-2}, \cite{Khovanskii-toroidal}) and is one of their motivations for introducing the notion of a convex chain. We give a sketch of a proof here.

\begin{prop}  \label{th-Euler-char-lb}
    For any toric line bundle  $\mathcal{L}$, and any character $u \in M$, we have: $\chi(\mathcal{L})_u=\alpha_\mathcal{L}(u)$.
\end{prop}
\begin{proof}[Sketch of proof]
One knows that the toric line bundle $\L$ on $X_\Sigma$ corresponds to a virtual polytope whose normal fan is refined by $\Sigma$. The convex chain $\alpha_\L$ of this virtual polytope, has a decomposition as an alternating sum of indicator functions of polyhedral cones (\cite[\S 4, Proposition 2]{Khovanskii-Pukhlikov-1}) as follows:
\begin{equation} \label{equ-alt-sum-virtual-polytope}
\alpha_\L = \sum_{\sigma \in \Sigma} (-1)^{\codim(\sigma)} \mathbbm{1}_{C_\sigma},
\end{equation}
where $C_\sigma$ is a shifted copy of the dual cone $\sigma^\vee$.
The above alternating sum is an extension of the classic Brianchon-Gram theorem to virtual polytopes. The Brianchon-Gram theorem states that the indicator function of a convex polytope is an alternating sum of the indicator functions of tangent cones at all the faces of the polytope (see \cite{BHS}).

On the other hand, the equivariant Euler characteristic $\chi(\L)_u$ can be computed using Cech cohomology with respect to the open cover by the affine toric charts $U_\sigma$, where the $\sigma$ are maximal cones in the fan. One verifies that this exactly coincides with 
\eqref{equ-alt-sum-virtual-polytope}.
This finishes the proof.
%
\end{proof}

The following lemma tells us that the cohomologies of a toric vector bundle do not change under refining the fan. Note that this is not the case for arbitrary (non-toric) varieties. The proof below was pointed to us by Amin Gholampour and Klaus Altmann.
\begin{lemma}  \label{lem-pull-back-inv}
The cohomologies and hence the equivariant Euler characteristic of a toric vector bundle do not change under pull-back by a toric birational morphism.     
\end{lemma}
\begin{proof}
Firstly, the results in \cite{Altmann} show that the $T$-weight spaces in cohomology groups $H^i(X_\Sigma, \E)$ can be computed from the combinatorial data of their so-called Weil decorations (see \cite[Theorem 1.5]{Altmann}). One also observes that the data of a Weil decoration does not change under refining the fan. It follows that the weight spaces in the cohomology groups do not change under refining the fan. 

Secondly, one can give a more general argument which does not use the toric setup. Let $f:X’\to X$ be a proper birational morphism such that $Rf_* \O_{X’} = \O_X$ (this is the case when $f$ is a toric birational morphism between toric varieties, see \cite[Theorem 9.2.5]{Cox-Little-Schenck}). Let $\E$ be a vector bundle on $X$. We claim that $H^i(X, \E) = H^i(X', f^*\E)$, for all $i$.
In fact let $p:X \to \Spec(\C)$ and $q:X' \to \Spec(\C)$ be the structure morphisms.
Then by the projection formula and the assumption $Rf_* \O_{X’} = \O_X$, we have $Rq_* f^*\E= Rp_*(Rf_* f^*\E)= Rp_* \E$.
Now the left-hand side computes the cohomologies $H^i(X', f^*\E)$ and the right hand side computes the cohomologies $H^i(X,\E)$.

\end{proof}

The following is the main result of this section. It generalizes Proposition \ref{th-Euler-char-lb} to toric vector bundles.
\begin{theorem}\label{thm-cc-eu}
Let $\E$ be a toric vector bundle on a complete toric variety $X_\Sigma$ with corresponding convex chain $\alpha_\E$. The equivariant Euler characteristic of $\E$ is given by
$$\chi(\E)_u = \alpha_\E(u), ~~~~~~~\forall u \in M.$$
In particular,
$$\chi(\E) = S(\alpha_\E) = \sum_{u\in M}\alpha_\E(u).$$    
\end{theorem}

\begin{proof}
    One knows that there is a refinement $\Sigma'$ of $\Sigma$ such that, if $f: X_{\Sigma'}\to X_\Sigma$ is the corresponding birational morphism, then the pull-back  $\E'=f^*\E$ has the same equivariant Chern classes (regarded as piecewise polynomial functions with respect to $\Sigma'$) and hence the same equivariant Euler characteristic, as a direct sum of (toric) line bundles $\mathcal{L}_1\oplus \mathcal{L}_2\cdots \oplus \mathcal{L}_r$ on $X_{\Sigma'}$.
    We have
    \begin{equation*}
        \chi(\E)_u=\chi(\E')_u=\chi(\bigoplus_{i=1}^r\L_i)_u=\sum_{i=1}^r\chi(\L_i)_u=\sum_{i=1}^r \alpha_{\mathcal{L}_i}(u)=\alpha_\E(u).
    \end{equation*}
\end{proof}


\section{Examples} \label{sec:example}
We give two examples here. For more details and computer codes see \cite{Huang-thesis}.

\begin{example}[Tangent bundle of $\P^2$]
    Let $\Sigma$ be the fan of the toric variety $\P^2$. The primitive vectors of rays are the vectors $\{\bm{v}_1, \bm{v}_2, \bm{v}_3\}$ where $\bm{v}_1$ and $\bm{v}_2$ are the standard basis and $\bm{v}_3=-\bm{v}_1-\bm{v}_2$. Let $\rho_i$ be the ray in $\Sigma$ generated by $\bm{v}_i$. 
    Consider the tangent bundle $\mathcal{T}_{\P^2}$ as a toric vector bundle. 
    We identify the fiber $E$ of the tangent bundle over the identity of the torus with $N\otimes_{\Z}\C={\C}^2$, $\bm{v}_1$ and $\bm{v}_2$ are also the standard basis of $E$.
    For each maximal cone $\sigma_i$ in the fan $\Sigma$, there is a basis of $E$ and a multiset of characters $\u(\sigma_i)$. And we have $\u(\sigma_1)=\{-\bm{w}_1, \bm{w}_2-\bm{w}_1\}$, $\u(\sigma_2)=\{\bm{w}_1-\bm{w}_2, -\bm{w}_2\}$ and $\u(\sigma_3)=\{\bm{w}_1, \bm{w}_2\}$.

Let ${\bf h} = (h_1 \leq h_2)$ be the multi-valued support function of $\mathcal{T}_{\P^2}$.
For $x=(x_1,x_2)\in N_{\R}$, we have
    \begin{equation*}
        h_1(x)=\begin{cases}
            \min_{\bm{w}\in u(\sigma_1)}\dotpro{\bm{w}}{x},&x\in\sigma_1\\
            \min_{\bm{w}\in u(\sigma_2)}\dotpro{\bm{w}}{x},&x\in\sigma_2\\
            \min_{\bm{w}\in u(\sigma_3)}\dotpro{\bm{w}}{x},&x\in\sigma_3\\
        \end{cases}=\begin{cases}
            -x_1,&x\in\sigma_1,\ x_2\ge0\\
            -x_1+x_2,&x\in\sigma_1,\ x_2<0\\
            x_1-x_2,&x\in\sigma_2,\ x_1\le0\\
            -x_2,&x\in\sigma_2,\ x_1>0\\
            x_2,&x\in\sigma_3,\ x_1\ge x_2\\
            x_1,&x\in\sigma_3,\ x_1 < x_2
        \end{cases}
    \end{equation*}
        \begin{equation*}
        h_2(x)=\begin{cases}
            \max_{\bm{w}\in u(\sigma_1)}\dotpro{\bm{w}}{x},&x\in\sigma_1\\
            \max_{\bm{w}\in u(\sigma_2)}\dotpro{\bm{w}}{x},&x\in\sigma_2\\
            \max_{\bm{w}\in u(\sigma_3)}\dotpro{\bm{w}}{x},&x\in\sigma_3\\
        \end{cases}=\begin{cases}
            -x_1+x_2,&x\in\sigma_1,\ x_2\ge0\\
            -x_1,&x\in\sigma_1,\ x_2<0\\
            -x_2,&x\in\sigma_2,\ x_1\le0\\
            x_1-x_2,&x\in\sigma_2,\ x_1>0\\
            x_1,&x\in\sigma_3,\ x_1\ge x_2\\
            x_2,&x\in\sigma_3,\ x_1 < x_2.
        \end{cases}
    \end{equation*}
    Consider
    \begin{equation*}
        h_s(x)=h_1(x)+h_2(x)=\begin{cases}
            -2x_1+x_2,&x\in\sigma_1\\
            x_1-2x_2,&x\in\sigma_2\\
            x_1+x_2,&x\in\sigma_3.
        \end{cases}
    \end{equation*}
    One sees that $h_2$ and $h_s$ are convex functions.
Let $P_2$ and $P_s$ be the polytopes whose support functions are $h_2$ and $h_s$ respectively. The polytopes $P_2$ and $P_s$ are depicted in Figure \ref{fig-polytope}.
    
    \begin{figure}[h]
        \centering
        \subfloat[\centering $P_s$]{{\includegraphics[width=6cm]{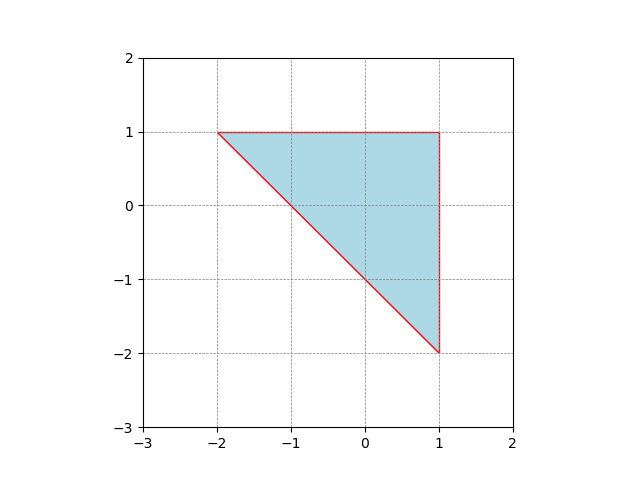} }}
        \qquad
        \subfloat[\centering $P_2$]{{\includegraphics[width=6cm]{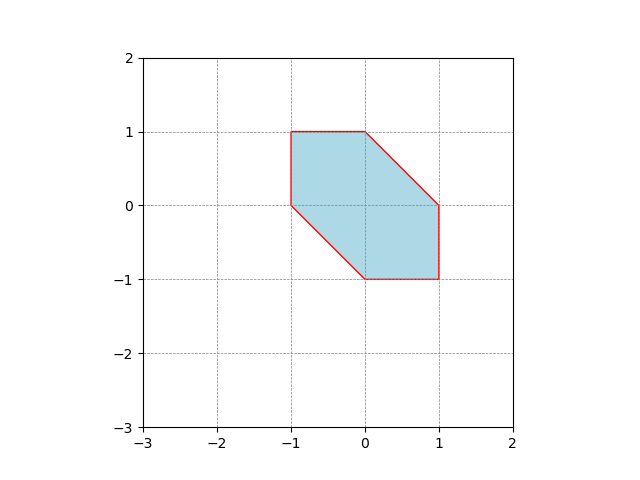} }}
        \caption{Polytopes from $h$} \label{fig-polytope}
    \end{figure}

    Since $h_1$ is not a convex function it only corresponds to a virtual polytope $P_1$. Since $h_1 = h_s-h_2$ and Theorem \ref{thm-vir-supp}, the virtual polytope $P_1$ equals $P_s\oplus(-P_2)$. Let $\alpha_1$ be the convex chain of this virtual polytope. Then (using Theorem \ref{th-Minkowski-inversion}) we have:
    \begin{equation*}
        \begin{split}
            \alpha_1 &=\mathbbm{1}_{P_s\oplus(-P_2)}=\mathbbm{1}_{P_s}*\mathbbm{1}_{P_2}^{-1}\\
            &=\mathbbm{1}_{P_s}*(-1)^{\dim P_2}\mathbbm{1}_{\interior(\symmetry P_2)}\\
            &=\mathbbm{1}_{P_s}*\sum_{\Delta\in\Gamma(\symmetry P_2)}(-1)^{\dim \Delta}\mathbbm{1}_\Delta\\
            &=\sum_{\Delta\in\Gamma(\symmetry P_2)}(-1)^{\dim \Delta}\mathbbm{1}_{P_s\oplus\Delta}.\\
        \end{split}
    \end{equation*}

    Let $\alpha = \alpha_1 +\mathbbm{1}_{P_2}$ be the convex chain associated to the toric vector bundle $\mathcal{T}_{\P^2}$. 
    The following figure gives the values of $\alpha$ on lattice points, where star at the origin represents value 2 and dot represents value 1.
    \begin{figure}[ht]
        \centering
        \includegraphics[scale=0.6]{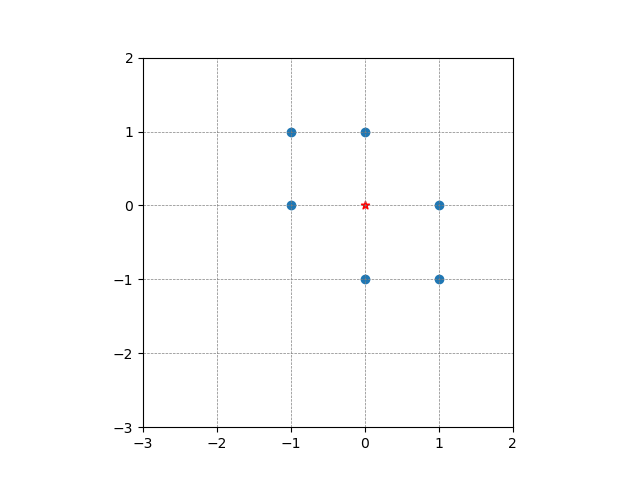}
        \caption{Values of the convex chain $\alpha$ on lattice points}
    \end{figure}
    This corresponds to
$$\dim H^0(\P^2,\mathcal{T}_{\P^2})_{(0, 0)} = \alpha((0, 0)) = 2,$$
and            
$$\dim H^0(\P^2,\mathcal{T}_{\P^2})_{u} = \alpha(u) = 1,$$
for $u=(-1, 1), (-1, 0), (0, 1), (0, -1), (1, 0), (1, -1)$.
\end{example}

\begin{example}
We consider a toric vector bundle $\mathcal{P}$ of rank 2 on the first Hirzebruch surface $X_\Sigma$ (see \cite[Example 4.4]{DJS}). The fan $\Sigma$ is given by the ray generators $\bm{v}_1=(1,0),\bm{v}_2=(0,1),\bm{v}_3=(-1,1),\bm{v}_4=(0,-1)$. Let $\rho_i$ be the ray in $\Sigma$ generated by $\bm{v}_i$ and $\sigma_{i,j}=\Span(\bm{v}_i, \bm{v}_j)$. The set $\{\bm{v}_1,\bm{v}_2\}$ is also a basis of the distinguished fiber $E$ of $\mathcal{P}$. The filtrations defining $\mathcal{P}$ are
    \begin{equation*}
        E^{\rho_1}(i)=\begin{cases}
            E,\ &i\le-2\\
            \Span(\bm{v}_1), \ &-2<i\le4\\
            0, &i>4
        \end{cases}
    \end{equation*}
    \begin{equation*}
        E^{\rho_2}(i)=\begin{cases}
            E,\ &i\le2\\
            \Span(\bm{v}_1), \ &2<i\le3\\
            0, &i>3
        \end{cases}
    \end{equation*}
    \begin{equation*}
        E^{\rho_3}(i)=\begin{cases}
            E,\ &i\le0\\
            \Span(\bm{v}_2), \ &0<i\le5\\
            0, &i>5
        \end{cases}
    \end{equation*}
    \begin{equation*}
        E^{\rho_4}(i)=\begin{cases}
            E,\ &i\le-1\\
            \Span(\bm{v}_1+\bm{v}_2), \ &-1<i\le3\\
            0, &i>3.
        \end{cases}
    \end{equation*}
The corresponding multiset of characters is given by 
    \begin{equation*}
        \begin{split}
            \u(\sigma_{1,2})&=\{-2\bm{w}_1+2\bm{w}_2, 4\bm{w}_1+3\bm{w}_2\}\\
            \u(\sigma_{2,3})&=\{-3\bm{w}_1+2\bm{w}_2, 3\bm{w}_1+3\bm{w}_2\}\\
            \u(\sigma_{3,4})&=\{-4\bm{w}_1+\bm{w}_2, -3\bm{w}_1-3\bm{w}_2\}\\
            \u(\sigma_{4,1})&=\{-2\bm{w}_1-3\bm{w}_2, 4\bm{w}_1+\bm{w}_2\}.
        \end{split}
    \end{equation*}

Let ${\bf h} = (h_1 \leq h_2)$ be the multi-valued support function corresponding to $\mathcal{P}$. For $x=(x_1,x_2)\in N_{\R}$, we have
    \begin{equation*}
        h_1(x)=\begin{cases}
            \min_{\bm{w}\in u(\sigma_{1,2})}\dotpro{\bm{w}}{x},&x\in\sigma_{1,2}\\
            \min_{\bm{w}\in u(\sigma_{2,3})}\dotpro{\bm{w}}{x},&x\in\sigma_{2,3}\\
            \min_{\bm{w}\in u(\sigma_{3,4})}\dotpro{\bm{w}}{x},&x\in\sigma_{3,4}\\
            \min_{\bm{w}\in u(\sigma_{4,1})}\dotpro{\bm{w}}{x},&x\in\sigma_{4,1}
        \end{cases}=\begin{cases}
            -2x_1+2x_2,&x\in\sigma_{1,2}\\
            -3x_1+2x_2,&x\in\sigma_{2,3},\ 6x_1+x_2\ge0\\
            3x_1+3x_2,&x\in\sigma_{2,3},\ 6x_1+x_2<0\\
            -3x_1-3x_2,&x\in\sigma_{3,4},\ x_1\le 4x_2\\
            -4x_1+x_2,&x\in\sigma_{3,4},\ x_1> 4x_2\\
            4x_1+x_2,&x\in\sigma_{4,1},\ 3x_1+2x_2\le 0\\
            -2x_1-3x_2,&x\in\sigma_{4,1},\ 3x_1+2x_2< 0
        \end{cases}
    \end{equation*}

    \begin{equation*}
        h_2(x)=\begin{cases}
            \max_{\bm{w}\in u(\sigma_{1,2})}\dotpro{\bm{w}}{x},&x\in\sigma_{1,2}\\
            \max_{\bm{w}\in u(\sigma_{2,3})}\dotpro{\bm{w}}{x},&x\in\sigma_{2,3}\\
            \max_{\bm{w}\in u(\sigma_{3,4})}\dotpro{\bm{w}}{x},&x\in\sigma_{3,4}\\
            \max_{\bm{w}\in u(\sigma_{4,1})}\dotpro{\bm{w}}{x},&x\in\sigma_{4,1}
        \end{cases}=\begin{cases}
            4x_1+3x_2,&x\in\sigma_{1,2}\\
            3x_1+3x_2,&x\in\sigma_{2,3},\ 6x_1+x_2\ge0\\
            -3x_1+2x_2,&x\in\sigma_{2,3},\ 6x_1+x_2<0\\
            -4x_1+x_2,&x\in\sigma_{3,4},\ x_1\le 4x_2\\
            -3x_1-3x_2,&x\in\sigma_{3,4},\ x_1> 4x_2\\
            -2x_1-3x_2,&x\in\sigma_{4,1},\ 3x_1+2x_2\le 0\\
            4x_1+x_2,&x\in\sigma_{4,1},\ 3x_1+2x_2< 0.
        \end{cases}
    \end{equation*}
    Also, we have 
    \begin{equation*}
        h_s(x)=h_1(x)+h_2(x)=
        \begin{cases}
            2x_1+5x_2,&x\in\sigma_{1,2}\\
            5x_2,&x\in\sigma_{2,3}\\
            -7x_1-2x_2,&x\in\sigma_{3,4}\\
            2x_1-2x_2,&x\in\sigma_{4,1}.
        \end{cases}
    \end{equation*}
One observes that $h_s$ and $h_2$ are convex. Figure \ref{fig-polytopes2} shows the corresponding polytopes $P_s$ and $P_2$.
    \begin{figure}[h]
        \centering
        \subfloat[\centering $P_s$]{{\includegraphics[width=6cm]{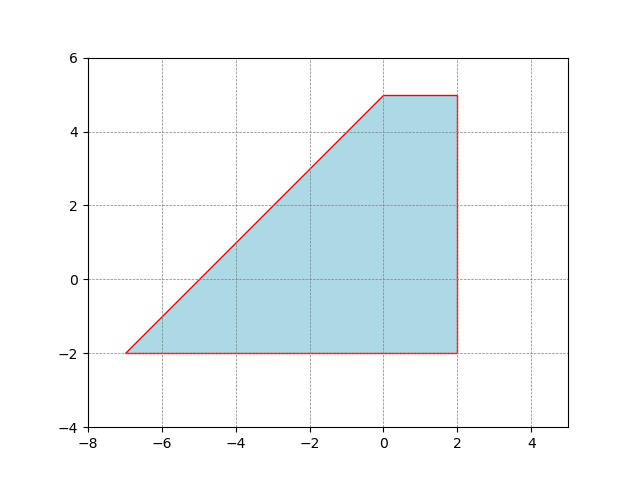} }}
        \qquad
        \subfloat[\centering $P_2$]{{\includegraphics[width=6cm]{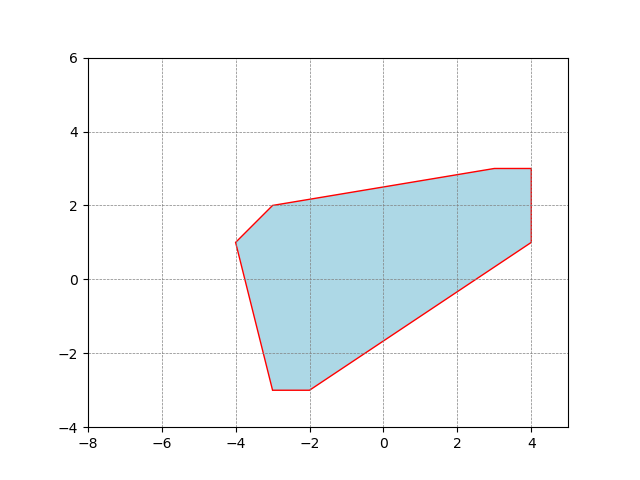} }}
        \caption{Polytopes from $h$} \label{fig-polytopes2}
    \end{figure}

Let $\alpha_1$ be the convex chain of the virtual polytope $P_1$ corresponding to the piecewise linear function $h_1$. Then $\alpha=\alpha_1+\mathbbm{1}_{P_2}$ is the convex chain associated to the toric vector bundle $\mathcal{P}$.
     The following figure gives the values of $\alpha$ on lattice points, where star at $(-1, 0)$ represents the value $-1$ and dot represents the value 1.
    \begin{figure}[ht]  
        \centering
        \includegraphics[scale=0.6]{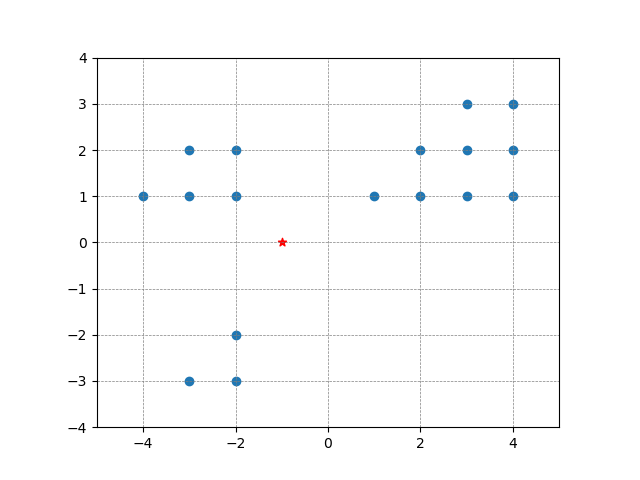}
        \caption{Values of the convex chain $\alpha$ on lattice points} \label{fig-cc}
    \end{figure}
    
    This corresponds to
$$            -\dim H^1(X,\mathcal{P})_{(-1, 0)} = -\alpha((-1, 0)) = -1,$$
and
$$\dim H^0(X,\mathcal{P})_{u} = \alpha(u) = 1,$$
    for the other lattice points $u$ marked by dots in Figure \ref{fig-cc}.
\end{example}

\end{document}